\documentclass [12pt] {article}

\title {Spaces of Types in Positive Model Theory}
\author {Levon Haykazyan}

\usepackage {amssymb}

\usepackage {amsthm}
\theoremstyle {definition}
\newtheorem {definition} {Definition} [section]
\newtheorem {theorem} [definition] {Theorem}
\newtheorem {proposition} [definition] {Proposition}
\newtheorem {lemma} [definition] {Lemma}
\newtheorem {corollary} [definition] {Corollary}
\newtheorem {fact} [definition] {Fact}
\newtheorem {example} [definition] {Example}
\newtheorem {remark} [definition] {Remark}

\usepackage {amsmath}
\DeclareMathOperator{\tp}{tp}
\DeclareMathOperator{\T}{T}
\DeclareMathOperator{\M}{M}
\DeclareMathOperator{\Spec}{Spec}
\DeclareMathOperator{\id}{id}
\renewcommand{\S}{\mathrm {S}}
\newcommand{\liff}{\leftrightarrow}
\DeclareMathOperator{\Sdiag}{\Sigma Diag}

\usepackage {natbib}

\usepackage {tikz}
\usetikzlibrary {cd}

\begin{document}
\maketitle

\begin{abstract}
We introduce a notion of the space of types in positive model theory based on
Stone duality for distributive lattices. We show that this space closely mirrors
the Stone space of types in the full first-order model theory with negation
(Tarskian model theory). We use this to generalise some classical results on
countable models from the Tarskian setting to positive model theory.
\end{abstract}

\section* {Introduction}

The collection of definable sets is the basic object of study in model theory.
Definable sets are the sets of solutions of formulas in a (universal) model. 
Tarskian model theory focuses on the collection of all first-order formulas,
i.e. formulas obtained from atomic one by taking finite conjunctions, finite
disjunctions, negation and existential quantification.  Under appropriate
conditions of stability and its variants, model theoretic methods can be used to
bring structure to the collection of definable sets. Oftentimes, however,
increased iterations of quantification can create more and more complex
definable sets.  And although these complex definable sets may be of little
interest themselves, their presence makes model theoretic analysis difficult or
impossible.

The solution to limit these extra definable sets, offered by the positive model
theory, is to only look at positive first-order formulas, i.e. those obtained
from the atomic ones using finite conjunctions, finite disjunctions and
existential quantification (note the absence of the negation). Positive model
theory generalises Tarskian model theory in the sense that the latter can be
treated in the former by changing the language. This process, commonly called
{\em Morleisation}, adds a predicate symbol for every first-order formula. So in
positive model theory one can have as much negation as one wishes. However, no
extra negation is forced by the framework.

Positive model theory has been introduced by \cite{benyaacov} and given its
present form by \cite{benyaacov-poizat}, although the ideas go back to Abraham
Robinson. It aims to generalise and apply the techniques of Tarskian model
theory to the more algebraic model theory of Robinson's school. Previous
important work with similar scope include \cite{shelah-lazy},
\cite{hrushovski-robinson} and \cite{pillay-ecforking}. Incidentally the same
formalism has been studied in topos theory under the names of {\em coherent
logic} and {\em finitary geometric logic} (see e.g.
\cite{maclane-moerdijk-sheaves}).

Let us fix a signature $\sigma$ and denote by $\Sigma$ the set of positive
formulas. In positive model theory one studies the class of positively closed
models of an h-inductive theory $T$ (see the next section for definitions.) In
analogy with Tarskian model theory one can look at the set $\M_n(T)$ of
$\Sigma$-types in $n$ variables realised in these models. This set turns out to
be precisely the set of maximal $\Sigma$-types (with respect to inclusion)
realised in all models. If we topologise $\M_n(T)$ by taking $\{\{p \in \M_n(T)
: \phi \in p \} : \phi \in \Sigma\}$ as a basis of closed sets, we end up with a
compact $\T_1$ space.

The long standing philosophy in model theory is that all relevant properties of
a theory can be recovered from the space of types
(\cite{poizat-yeshkeyev-positive} call such properties {\em semantic}). More
precisely, we can recover a theory $T'$ from the data of the spaces of types and
appropriate maps between them. This theory is (up to interpretation) the {\em
infinite Morleisation} of $T$, where we add an atomic formula for each
$\Sigma$-type-definable set. Therefore, in such an approach to positive model
theory, there is no distinction between definable and type-definable sets. This
aspect may be part of the design and can be viewed as an advantage. But it also
means that positive model theory, in this approach, does not truly generalise
Tarskian model theory, where the distinction between definable and
type-definable sets is present. So it can also be viewed as a disadvantage.

To rectify this we instead look at the set $\S_n(T)$ of {\em all} $\Sigma$-types
realised in models of $T$ (as opposed to just the maximal ones). We topologise
$\S_n(T)$ by taking $\{\{p \in \S_n(T): \phi \in p\} : \phi \in \Sigma\}$ to be
a basis of open sets. (In contrast in the topology on $\M_n(T)$ these are basic
closed sets. However for $\S_n(T)$ both topologies can be used to recover the
other one and the topology we chose is more suitable for our applications.) The
space $\S_n(T)$ so defined will be a spectral space (i.e. homeomorphic to the
space of prime ideals of a commutative ring.) Conversely any spectral space can
arise in this way. We prove that these spaces of types determine the theory up
to interpretation (Theorem \ref{equivtheorem}). This means that from these
spaces we can recover the Morleisation of a theory as opposed to infinite
Morleisation.

It is worth stressing that the above definition of $\S_n(T)$ is just a concrete
realisation of Stone duality for distributive lattices. Although its definition
refers to all models of $T$, our aim is use it in the study of positively
closed models of $T$. We show that $\S_n(T)$ closely parallels the space of
types in Tarskian model theory. In particular we prove that in a countable
signature any meagre set of types can be omitted (Theorem
\ref{omittingmeagretypes}). We use this to connect the existence of atomic
models and countable categoricity to the topology of the spaces of types
(Theorems \ref{countcat}, \ref{denseexists}, \ref{countableexists}). The
formulations of these results are direct generalisations of classical results to
the non-Hausdorff topology of $\S_n(T)$. They are however not more difficult to
obtain, than their Tarskian counterparts. This suggests to us that positive
model theory is a more natural setting for these results.

\section {Preliminaries}

\subsection {Positive Model Theory}

In this subsection we summarise the necessary background from positive model
theory. All the results stated here can be found in \cite{benyaacov-poizat}
whose conventions we follow with minor modifications. The reader can also
consult \cite{belkasmi} or \cite{poizat-yeshkeyev-positive} for basics on
positive model theory.

Fix a signature $\sigma$. In positive model theory one looks at the formulas
obtained from atomic formulas using the connectives $\land, \lor$ and the
quantifier $\exists$. We denote the set of such formulas by $\Sigma$. We also
stipulate that the language includes the propositional constant $\bot$ for
falsehood. This is needed in order to positively define the empty set. The
whole universe can be positively defined by the formula $\bar x = \bar x$, but
we sometimes use $\top$ for it. We denote by $\Pi$ the set of negations of
positive formulas, i.e. $\Pi = \{\lnot \phi : \phi \in \Sigma\}$.

A function $f : M \to N$ between $\sigma$-structure is called a {\em
homomorphism} if for every atomic $\phi(\bar x)$ and every $\bar m \in M$ it is
the case that $M \models \phi(\bar m)$ implies $N \models \phi(f(\bar m))$. Note
that the same would be true for every $\phi \in \Sigma$. In general a
homomorphism need not be injective, since $N \models f(a) = f(b)$ does not imply
$M \models a = b$. If there is a homomorphism from $M$ to $N$, we say that $N$
is a {\em continuation} of $M$. A homomorphism $f : M \to N$ is called an {\em
immersion} if for every $\phi \in \Sigma$ and $\bar m \in M$ it is the case that
$M \models \phi(\bar m)$ if and only if $N \models \phi(f(\bar m))$.

The category of $\sigma$-structures and homomorphisms admits inductive limits in
the following sense. Let $(I, \le)$ be a totally ordered set. Assume that we
have families $(M_i : i \in I)$ of structure and of homomorphisms $(h_{ij} : M_i
\to M_j : i \le j \in I)$ with the following properties:
\begin{itemize}
\item $h_{ii}$ is the identity for every $i \in I$,
\item $h_{ik} = h_{jk} \circ h_{ij}$ for every $i \le j \le k \in I$.
\end{itemize}
Then we can naturally define a structure with universe $\bigsqcup_{i \in I}
M_i/\sim$ where the $\sim$ is defined by $a_i \sim a_j$ (where $a_i \in M_i$ and
$a_j \in M_j$) iff $h_{ij}(a_i) = a_j$. (We leave the details of interpreting
the symbols of $\sigma$ to the reader.) This structure is denoted by
$\lim\limits_{\to} M_i$ and called the inductive limit of the family. In the
category-theoretic language $\lim\limits_{\to} M_i$ is the colimit of a diagram
of shape $I$ where $I$ is viewed as a small category.

A class $\mathcal C$ of $\sigma$-structures is called {\em h-inductive} if it is
closed under inductive limits. The letter `h' stands for the word
`homomorphism'. The unfortunate practice of prefixing the terms with h to avoid
confusion with the existing related terminology appears to be the norm now.

In an h-inductive class we can find special structures.

\begin{definition}
A structure $M$ is called {\em positively closed} in a class $\mathcal C$ if
for every $N \in \mathcal C$, every homomorphism $f : M \to N$ is an immersion.
\end{definition}

This is related to Robinson's notion of existential closure, which we recover as
a special case later on. The classical argument for the construction of
existentially closed structures generalises to our setting.

\begin{fact}
In an h-inductive class $\mathcal C$ every structure $M \in \mathcal C$ can be
continued to a positively closed one.
\end{fact}

In model theory we are interested in the classes that can be axiomatised. To
characterise h-inductive classes within those we need the notion of an
h-inductive sentence.

\begin{definition}
A sentence of the form $\forall \bar x (\phi(\bar x) \to \psi(\bar x))$ where
$\phi, \psi \in \Sigma$ is called {\em h-inductive}.
\end{definition}

By taking $\phi = \top$ we can see that $\Sigma$-sentences are all (equivalent
to sentences that are) h-inductive. Similarly by taking $\psi = \bot$ we see
that so are $\Pi$-sentences.

\begin{fact}
An elementary class is h-inductive if and only if it can be axiomatised by a set
of h-inductive sentences. We call such a set an h-inductive theory.
\end{fact}

We can describe positive model theory very roughly as the study of positively
closed models of an h-inductive theory.

A type is a set of formulas consistent with a given theory. If all formulas in a
type are from say $\Sigma$, we call it a $\Sigma$-type. Unlike full types,
$\Sigma$-types of tuples need not be maximal, i.e. it is possible to have
$\tp_{\Sigma}(\bar a) \subsetneq \tp_{\Sigma}(\bar b)$. Positive closure of a
model can be characterised by the types that tuples from this model realise.

\begin{fact}
\label{pc-fact}
Let $T$ be an h-inductive theory and $M \models T$. Then the following
conditions are equivalent.
\begin{itemize}
\item The model $M$ is positively closed.
\item For every $\bar a \in M$ the $\Sigma$-type $\tp_\Sigma(\bar a)$ is maximal
with respect to inclusion.
\item For every $\phi(\bar x) \in \Sigma$ and $\bar a \in M$ such that $M
\models \lnot \phi(\bar a)$ there is a formula $\psi(\bar x) \in \Sigma$ such
that $M \models \psi(\bar a)$ and $T \models \lnot \exists \bar x (\phi(\bar x)
\land \psi(\bar x))$.
\end{itemize}
\end{fact}

Note that the last condition depends only on negative consequences of $T$. We
denote the set of negative consequences of $T$ by $T|_\Pi$ and call its {\em
negative part}.

\begin{fact}
\label{companion-fact}
An h-inductive theory and its negative part have exactly the same positively
closed models.
\end{fact}

An h-inductive theory $T$ has the {\em joint continuation property} if for any
two models $M_1, M_2 \models T$, there is a model $N \models T$ and
homomorphisms $f_i : M_i \to N$. The joint continuation property plays the role
of the completeness in Tarskian model theory. In particular if $T$ has the joint
continuation property, then any two positively closed models satisfy the same
positive and negative sentences. We also have the following characterisation.

\begin{fact}
An h-inductive theory $T$ has the joint continuation property if and only if for
any $\Pi$-sentences $\phi$ and $\psi$, if $T \models \phi \lor \psi$, then $T
\models \phi$ or $T \models \psi$
\end{fact}

If we have an arbitrary first-order theory $T$ (with negation) we can treat it
in positive model theory by considering its Morleisation $T'$. This is a theory
in a new signature $\sigma'$ where we add an $n$-ary relation symbol $R_\phi$
for every first-order formula $\phi(\bar x)$ in $n$ free variables in the
signature of $T$. The theory $T'$ then expresses how $R_\phi$ is interpreted by
containing sentences of the form $\forall \bar x (R_{\phi \lor \psi}(\bar x)
\liff R_\phi(\bar x) \lor R_\psi(\bar x))$ and $\forall \bar x (R_{\lnot
\phi}(\bar x) \liff \lnot R_\phi(\bar x))$, etc.  Actually $T'$ can be viewed as
an h-inductive theory if we replace the sentences of the last form by the
following pair of sentences:
$$\forall \bar x (\top\to R_\phi(\bar x) \lor R_{\lnot \phi}(\bar x)),$$
$$\forall \bar x (R_\phi(\bar x) \land R_{\lnot \phi}(\bar x) \to \bot).$$
Then $T$ and $T'$ have the same models. The $\sigma'$-homomorphisms are
precisely the $\sigma$-elementary embeddings, and all models of $T'$ are
positively closed. More generally we call an h-inductive theory $T$ {\em
positively model-complete} if $\Sigma$ is closed under complementation modulo
$T$: i.e. for every formula $\phi(\bar x) \in \Sigma$ there is a formula
$\psi(\bar x) \in \Sigma$ such that $T \models \forall \bar x (\lnot \phi(\bar
x) \liff \psi(\bar x))$. The homomorphisms in this case are elementary
embeddings and all models are positively closed. All our results specialise to
classical results in this setting.

The second important class of h-inductive theories is obtained if we require
only atomic formulas to have negations (modulo the theory) in $\Sigma$. In that
case homomorphisms become embeddings and a model is positively closed if and
only if it is existentially closed. These notions have been extensively studied
by Robinson and his school. The omitting types theorem in this setting is well
known (see e.g.  \cite{hodges-games}). Atomic existentially closed models have
been studied by \cite{simmons-large-small} using different methods.

Thus positive model theory generalises both Tarskian and Robinsonian model
theories.\footnote{In North America these two used be called western and eastern
model theories respectively.} It is therefore a very general framework for model
theoretic study of structures.

\subsection {Topology}

In this paper we will deal with non-Hausdorff spaces. We gather here some
terminology and results that may not be familiar to the reader.

All topological spaces in this paper will satisfy the $\T_0$ separation axiom
(also called {\em Kolmogorov spaces}). This means that for every two points,
there is an open set containing exactly one of them. We will say that a space is
{\em compact} if every open cover has a finite subcover (without requiring the
space to be Hausdorff). In some mathematical cultures these spaces are known as
{\em quasi-compact}.

A topological space is called {\em irreducible} (also called {\em
hyperconnected}) if it cannot be written as a union of two proper closed sets
(equivalently any two nonempty open sets intersect nontrivially). Maximal
irreducible subspaces of a space are called its {\em irreducible components}.
Irreducible components are always closed and their union is the whole space.
Unlike connected components, however, irreducible components need not be
disjoint. In a Hausdorff space the only nonempty irreducible sets are
singletons, and so these are the irreducible components of the space.

We denote the topological closure of a subset $Y$ of a space $X$ by $\overline
Y$. For a singleton $Y = \{y\}$, we denote $\overline Y$ by $\overline y$. A
point $x$ of the topological space $X$ is called {\em generic} if $\overline x =
X$.

\begin{definition}
A topological space is called {\em sober} if every irreducible closed subspace
has a unique generic point.
\end{definition}

All Hausdorff spaces are sober.

Recall that a topological space is called a {\em Baire space} if the
intersection of a countable collection of dense open sets is dense. We will use
the following generalisation of Baire Category Theorem to sober spaces.

\begin{fact}[\cite{isbell}, see also \cite{hofmann}]
\label{isbell-category}
Every compact sober space is a Baire space.
\end{fact}

In the above references one can find stronger versions of this theorem. However
the weaker version is sufficient for our purposes.

Recall that a subset $Y$ of a topological space is called {\em nowhere dense} if
$\overline Y$ has empty interior. A set that is not nowhere dense is called {\em
somewhere dense}. A {\em meagre} set is a countable union of nowhere dense sets.
And a complement of a meagre set is called {\em comeagre}. Note that any dense
open set is comeagre and so is a countable intersection of comeagre sets. In
Baire spaces being comeagre is a largeness property: any comeagre set is dense
and in particular nonempty.

\section {The Space of Types}

Let $\sigma$ be a signature, $\Sigma$ be the set of its positive formulas and
$\Pi$ be the set of its negative formulas. Consider an h-inductive theory $T$.

\begin{definition}
For each $n \ge 0$ let $\S_n(T)$ be the set of complete $\Sigma \cup \Pi$-types
of $T$ in variables $\bar x = (x_0, ..., x_{n-1})$. For a formula $\phi(\bar x)$
let $[\phi] = \{p \in \S_n(T) : \phi \in p\}$. We topologise $\S_n(T)$ by taking
$\{[\phi] : \phi \in \Sigma\}$ as a basis of open sets (equivalently $\{[\phi] :
\phi \in \Pi\}$ will be a basis of closed sets). We call $\S_n(T)$ the Stone
space of $T$ in $n$ variables.
\end{definition}

\begin{remark}
There is a way of arriving at $\S_n(T)$ without using $\Pi$ at all via Stone
duality for distributive lattices. However considering $\Pi$-formulas explicitly
makes the construction model theoretically more transparent. We now proceed to
explain the connection to Stone duality.
\end{remark}

For each $n$ consider the set $\Sigma_n$ of positive formulas in $\bar x = (x_0,
..., x_{n-1})$. We identify two such formulas $\phi(\bar x)$ and $\psi(\bar x)$
if $T \models \forall \bar x (\phi(\bar x) \liff \psi(\bar x))$. The set
$\Sigma_n$ has a natural structure of a bounded distributive lattice. There is a
bijection between $\S_n(T)$ and the set of prime filters on $\Sigma_n$.

Recall that a {\em filter} $p \subsetneq \Sigma_n$ is a nonempty proper subset
such that
\begin{itemize}
\item $\phi(\bar x) \in p$ and $\phi(\bar x) \to \psi(\bar x)$ implies that
$\psi(\bar x) \in p$;
\item $\phi(\bar x), \psi(\bar x) \in p$ implies that $\phi(\bar x) \land
\psi(\bar x) \in p$.
\end{itemize}
It is called {\em prime} if in addition
\begin{itemize}
\item $\phi(\bar x) \lor \psi(\bar x) \in p$ implies that $\phi(\bar x) \in p$
or $\psi(\bar x) \in p$.
\end{itemize}
Note that in a boolean algebra every prime filter is maximal. This is no longer
the case for a general distributive lattice. 

\begin{proposition}
Restriction to $\Sigma$-formulas is a bijection from $\S_n(T)$ to the set of
prime filters on $\Sigma_n$.
\end{proposition}

\begin{proof}
Given $\bar a \in M \models T$, the $\Sigma$-type $\tp_{\Sigma}(\bar a) =
\{\phi(\bar x) \in \Sigma_n : M \models \phi(\bar a)\}$ is a prime filter.
Conversely given a prime filter $p$, the set $p \cup \{\lnot \phi(\bar x): \phi
\in \Sigma_n \setminus p\}$ is consistent.  Indeed, otherwise there are
$\phi_1(\bar x), ..., \phi_n(\bar x) \in p$ and $\psi_1(\bar x), ...,
\psi_m(\bar x) \in \Sigma_n \setminus p$ such that 
$$T \models \forall \bar x [(\phi_1(\bar x) \land ... \land \phi_n(\bar x)) \to
(\psi_1(\bar x) \lor ... \lor \psi_m(\bar x))].$$
But then $\phi_1(\bar x) \land ... \land \phi_n(\bar x) \in p$ and so
$\psi_1(\bar x) \lor ...  \lor \psi_m(\bar x) \in p$. Since $p$ is prime, it
implies that $\psi_i(\bar x) \in p$ for some $i = 1, ..., m$.
\end{proof}

We can recover $\Sigma_n$ from the topology on $\S_n(T)$ as follows.

\begin{lemma}
The set $[\phi(\bar x)]$ is a compact open subset of $\S_n(T)$. Conversely if
$O \subseteq \S_n(T)$ is compact open, then $O = [\phi(\bar x)]$ for some
$\phi \in \Sigma$.
\end{lemma}

\begin{proof}
The set $[\phi]$ is open by definition. Now let $\{[\psi_i] : i \in I\}$ be a
cover of $[\phi]$ by basic open subsets. Assume for contradiction that there is
no finite subcover. This means that for every finite $I_0 \subseteq I$ there a
type $p \in [\phi] \setminus \bigcup_{i \in I_0} [\psi_i]$ which is to say that
$\{\phi(\bar x)\} \cup \{\lnot \psi_i(\bar x) : i \in I_0\}$ is consistent. But
then by the compactness theorem $\{\phi(\bar x)\} \cup \{\lnot \psi_i(\bar x) :
i \in I\}$ is consistent contradicting the fact that $\{[\psi_i] : i \in I\}$
covers $[\phi]$.

Conversely assume that $O \subseteq \S_n(T)$ is compact and open. Then $O$ is a
union of basic open sets and by compactness we can take this to be a finite
union. Say $O = [\phi_0] \cup ... \cup [\phi_{n-1}] = [\phi_0 \lor ... \lor
\phi_{n-1}]$ and $\phi_0 \lor ... \lor \phi_{n-1} \in \Sigma$.
\end{proof}

To characterise the properties of $\S_n(T)$ we need the following definition.

\begin{definition} 
A topological space $X$ is called {\em spectral} if the following conditions
hold.
\begin{itemize}
\item $X$ is compact and $\T_0$.
\item The compact open subsets form a basis of the topology.
\item A finite intersection of compact open subsets is compact (and of course
open).
\item $X$ is sober.
\end{itemize}
A continuous function $f : X \to Y$ between spectral spaces is called {\em
spectral} if the preimage of a compact open set is compact (and open).
\end{definition}

Incidentally a space is spectral if and only if it is homeomorphic to $\Spec R$
for some commutative ring $R$ in Zariski topology. And a ring homomorphism $f :
R \to S$ induces a spectral map $f : \Spec R \to \Spec S$ (see
\cite{hochster-spec}).

\begin{fact} [\cite{stone-dlattice}]
The space of prime filters of a bounded distributive lattice $D$ is a spectral
space. Conversely compact open subsets of a spectral space $X$ form a bounded
distributive lattice. These two operations are inverses of each others. That is,
the lattice of compact open subsets of the space of prime filters of $D$ is
isomorphic to $D$. And the space of prime filters of the lattice of compact open
sets of $X$ is isomorphic to $X$.
\end{fact}

A closed set in $\S_n(T)$ is a set of the form $[p(\bar x)]$ where $p$ is a set
of $\Pi$-formulas. It will be irreducible just if $p \vdash \phi(\bar x) \lor
\psi(\bar x)$ implies $p \vdash \phi(\bar x)$ or $p \vdash \psi(\bar x)$ for
$\phi, \psi \in \Pi$. In such a case $\{\phi(\bar x) \in \Pi : p \vdash
\phi(\bar x)\} \cup \{\psi(\bar x) \in \Sigma : p \not \vdash \lnot \psi(\bar
x)\}$ is its generic point.

\begin{remark}
Our notion of the space of types differs from that of \cite{benyaacov}. The
essential difference is that \cite{benyaacov} considers only maximal
$\Sigma$-types or equivalently maximal filters of the lattice $\Sigma_n$.
However it is not possible to recover the lattice from its space of maximal
filters. That is the reason why we redefine the space of types as the set of
prime filters. Denoting this space with $\S_n$ clashes with the terminology of
Ben-Yaacov. However, given the connection with Stone duality and the tradition
in model theory, we find it appropriate to call it the Stone space and denote it
by $\S_n(T)$.

Another difference is that we topologise $\S_n$ by taking $\{[\phi] : \phi \in
\Sigma\}$ as the basis of open rather than closed sets (as is done in
\cite{benyaacov}). This difference is not essential, as both topologies can be
recovered from one another. In topological terms, the two spaces are Hochster
duals of each other (see \cite{hochster-spec}). But this is straightforward in
terms of lattices: the dual topology is just the Stone topology of the opposite
lattice of $\Sigma_n$, which is (isomorphic to) the lattice of $\Pi$-formulas in
$n$ variables. Therefore both the Stone space and its dual are spectral spaces.
Both spaces reflect various aspects of $T$. For example the maximal
$\Sigma$-types of $T$ will correspond to closed points of the dual space and the
irreducible components of the Stone space. For applications in this paper we use
the Stone topology exclusively. However it seems likely that in some
applications the dual topology is the natural choice.
\end{remark}

We can use the space $\S_n(T)$ to characterise some properties of $T$.
\begin{proposition}
An h-inductive theory $T$ is positively model-complete (i.e. $\Sigma$ is closed
under negation modulo $T$) if and only if for every $n$ the space $\S_n(T)$ is
Hausdorff.
\end{proposition}

\begin{proof}
Assume $T$ is positively model-complete. Pick two points $p, q \in \S_n(T)$.
Then there is a $\Sigma$-formula $\phi$ in one, but not the other. Say, without
loss of generality, $\phi \in p \setminus q$. There is a formula $\psi \in
\Sigma$ that is equivalent to $\lnot \phi$ modulo $T$. Then $\psi \in q$ and
$[\phi] \cap [\psi] = \emptyset$.

Conversely assume that $\S_n(T)$ is Hausdorff for every $n$. Let $\phi \in
\Sigma$ be a formula. Then $[\phi]$ is compact and open in $\S_n(T)$. Since
$\S_n(T)$ is Hausdorff $[\phi]$ is also closed. Hence $\S_n(T) \setminus [\phi]$
is clopen and compact (since $\S_n(T)$ is compact). So there is $\psi \in
\Sigma$ such that $\S_n(T) \setminus [\phi] = [\psi]$. Then $\psi$ is equivalent
to $\lnot \phi$ modulo $T$.
\end{proof}

We say that $T$ has the {\em amalgamation property for homomorphisms} if given
three models $M_0, M_1, M_2 \models T$ and homomorphisms $f_1 : M_0 \to M_1$ and
$f_2 : M_0 \to M_2$ there is a model $N \models T$ and homomorphisms $g_1 : M_1
\to N$ and $g_2 : M_2 \to N$ such that $g_1f_1 = g_2f_2$.

\begin{proposition}
An h-inductive theory $T$ has the amalgamation property for homomorphisms if and
only if the irreducible components of $\S_n(T)$ are disjoint for every $n$.
\end{proposition}

\begin{proof}
Assume that the irreducible components of $\S_n(T)$ are disjoint for every $n$.
Let $M_0, M_1, M_2 \models T$ be three models and $f_1 : M_0 \to M_1$ and $f_2
: M_0 \to M_2$ be homomorphisms. Add new constant symbols for elements of $M_0,
M_1$ and $M_2$. Let $\Sdiag(M_i)$ denote the set of all $\Sigma$-sentences true
in $M_i$ in this new language. Thus if $N \models \Sdiag(M_i)$, then we can
construct a homomorphism from $M_i$ to $N$ by taking $m \in M_i$ to its
interpretation in $N$. To ensure that these homomorphisms agree on images of
$M_0$ we will further need that $N \models f_1(m) = f_2(m)$ for every $m \in
M_0$. Thus we need to show that
$$T \cup \Sdiag(M_1) \cup \Sdiag(M_2) \cup \{f_1(m) = f_2(m) : m \in M_0\}$$
is consistent. Assume not. Then by compactness (and existentially quantifying
out additional parameters) there are formulas $\phi_1(f_1(\bar m)) \in
\Sdiag(M_1)$ and $\phi_2(f_2(\bar m)) \in \Sdiag(M_2)$ such that the conjunction
$\phi_1(\bar x) \land \phi_2(\bar x)$ is inconsistent with $T$. But then
$\tp_\Sigma^{M_0}(\bar m) \subseteq \tp_\Sigma^{M_i}(f_i(\bar m))$ for $i = 1,
2$. However $\tp_{\Sigma \cup \Pi}^{M_i}(f_i(\bar m))$ can be separated by open
sets $[\phi_i]$ and so lay in distinct irreducible components. But any such
component must contain $\tp_{\Sigma \cup \Pi}^{M_0}(\bar m)$ showing that they
are not disjoint.

Conversely assume that there is a type $p(\bar x) \in \S_n(T)$ that is in two
irreducible components. Let $q_1(\bar x)$ and $q_2(\bar x)$ be the generic types
of these components. Let $M_0 \models T$ be a model that realises $p$, say $\bar
m \models p(\bar x)$. By compactness $T \cup \Sdiag(M_0) \cup q_i(\bar m)$ is
consistent. This gives us two models $M_1$ and $M_2$ and two homomorphisms $f_1
: M_0 \to M_1$ and $f_2 : M_0 \to M_2$ which cannot be amalgamated.
\end{proof}

\begin{proposition}
Restriction to $\Pi$-part is a bijection from $\S_0(T)$ to the set of
$\Pi$-theories with the joint continuation property that extend $T|_\Pi$. (Here
we identify two theories if they are consequences of one another.)
\end{proposition}

\begin{proof}
Let $p \in \S_0(T)$. Then there is $M \models T$ such that $\{\phi \in \Pi : M
\models \phi\} = p|_\Pi$. So $T|_\Pi \subseteq p|_\Pi$. Now let $\phi, \psi \in
\Pi$ be such that $\phi \lor \psi \in p|_\Pi$. Then $M \models \phi \lor \psi$
and so $M \models \phi$ or $M \models \psi$. Therefore $\phi \in p|_\Pi$ or
$\psi \in p|_\Pi$ showing that $p|_\Pi$ has the joint continuation property.

Conversely let $T'$ be a $\Pi$-theory with the joint continuation property
extending $T|_\Pi$. Let $M \models T'$ be positively closed (in the class of
models of $T'$). Then we claim that for every negative formula $\phi$ it is the
case that $M \models \phi$ iff $T' \models \phi$. The right to left implication
is clear.  So let us show the other direction. Assume $T' \not \models \phi$.
Let $M' \models T' \cup \{\lnot \phi\}$. By the joint continuation property,
there is a model $N$ that continues both $M$ and $M'$. But since $\lnot \phi$ is
positive, it follows that $N \models \lnot \phi$. Therefore $M \models \lnot
\phi$ by positive closure.
\end{proof}

\section {The Type Space Functor and Interpretations of Theories}

From $\S_n(T)$ we can recover the $\Sigma$-formulas in $n$ variables. In order
to recover the relationship between $n$-variable formulas and $m$-variable
formulas we need to look at maps between $\S_n(T)$ and $\S_m(T)$.

Throughout this section we use natural numbers as set-theoretic ordinals, i.e.
each natural numbers is the set of its predecessors. Let $f : n \to m$ be a
function. It induces a function $f^* : \S_m(T) \to \S_n(T)$ as follows. Let $p
\in S_m(T)$ and pick a realisation $(a_0, ..., a_{m-1})$ in some model. Then map
$p$ to the type $\tp_{\Sigma \cup \Pi}(a_{f(0)},...,a_{f(n-1)})$. Equivalently
$\phi(y_0, ..., y_{n-1}) \in f(p)$ if and only if $\phi(x_{f(0)}, ...,
x_{f(n-1)}) \in p$. This assignment is functorial (i.e. $(gf)^* = f^*g^*$) and
our goal is to establish the properties of this functor.

\begin{proposition}
For any function $f : n \to m$ the associated map $f^*: \S_m(T) \to \S_n(T)$
is spectral and open.
\end{proposition}

\begin{proof}
It is easy to verify that
\begin{align*}
q(y_0, ..., y_{n-1}) & \in f([\psi(x_0, ..., x_{m-1})]) \text { iff } \\ & 
\exists x_0, ..., x_{m-1} (\psi(\bar x) \land y_0 = x_{f(0)} \land ... \land
y_{n-1} = x_{f(n-1)}) \in q
\end{align*}
and
$$p(x_0, ..., x_{m-1}) \in f^{-1}([\phi(y_0, ..., y_{m-1}]) \text { iff }
\phi(x_{f(0)}, ..., x_{f(m-1)}) \in p.$$
This shows that images and preimages of compact open sets are compact open.
\end{proof}

Now we gather all these data in a contravariant functor.

\begin{definition}
Let $\mathbf {FinOrd}$ be the category of finite ordinals and functions. Let
$\mathbf {Spec}$ be the category of spectral spaces and spectral open maps. To
an h-inductive theory $T$ we associate a contravariant functor $\S(T) : \mathbf
{FinOrd} \to \mathbf {Spec}$ that takes $n$ to $\S_n(T)$ and $f : n \to m$ to
$f^* : \S_m(T) \to \S_n(T)$. We call $\S(T)$ the {\em type space functor} of
$T$.  \end{definition}

We want to characterise theories that have naturally isomorphic type space
functors. Intuitively this will be the case if and only if the two theories have
the same definable sets. Expressing this formally, however, is not
straightforward. For the sake of exposition we make the following simplifying
assumption: {\em all signatures are assumed to be purely relational}. There is
no material loss of generality, since the standard tricks of replacing functions
and constants by relations applies to positive model theory (see e.g.
\cite{poizat-yeshkeyev-positive}).

\begin{definition}
Let $T$ and $T'$ be two h-inductive theories in (relational) signatures $\sigma$
and $\sigma'$ respectively. An {\em interpretation} $\Gamma$ of $T$ in $T'$
associates to each $n$-ary relation symbol $R \in \sigma$ a formula $\Gamma(R)
\in \Sigma'$ with $n$ free variables, such that the following holds. If $M'
\models T'$, then the $\sigma$-structure $\Gamma^*(M')$ with universe $M'$ and
each symbol $R$ interpreted by the set defined by $\Gamma(R)(\bar x)$ in $M'$ is
a model of $T$.
\end{definition}

The association of $\sigma'$-formulas to relation symbols of $\sigma$ given by
$\Gamma$ naturally extends to all $\sigma$-formulas. This will take
$\Sigma$-formulas to $\Sigma'$-formulas, $\Pi$-formulas to $\Pi'$-formulas, etc.
Note that the condition $\Gamma^*(M') \models T$ for every model $M' \models T'$
can be equivalently stated as $T' \models \forall \bar x (\Gamma(\phi)(\bar x)
\to \Gamma(\psi)(\bar x))$ for every h-inductive axiom $\forall \bar x
(\phi(\bar x) \to \psi(\bar x))$ of $T$.

We can compose two interpretations in the natural way. The identity
interpretation just associates $R(\bar x)$ to the relation symbol $R$. Thus
theories with interpretations form a category. An interpretation $\Gamma$ of $T$
in $T'$ is an isomorphism if and only if it is bijective from $\Sigma$ to
$\Sigma'$ (when we identify equivalent formulas module $T$ and $T'$
respectively) and $T \models \chi$ iff $T' \models \Gamma(\chi)$ for every
h-inductive sentence $\chi$.

\begin{theorem}
\label{equivtheorem}
Two h-inductive theories $T$ and $T'$ are isomorphic if and only if $\S(T)$ and
$\S(T')$ are naturally isomorphic.
\end{theorem}

\begin{proof}
Let $\Gamma$ be an interpretation of $T$ into $T'$ that has an inverse. The
interpretation $\Gamma$ induces a function $\beta_n : \S_n(T') \to \S_n(T)$ as
follows.  Let $\bar a' \in M' \models T'$ be a realisation of $p' \in \S_n(T')$.
Then $\Gamma^*(M') \models T$ and we can define $\beta_n(p')$ to be the type of
$\bar a'$ in $\Gamma^*(M')$. Equivalently $\beta_n(p') = \{\phi(\bar x) \in
\Sigma \cup \Pi : \Gamma(\phi)(\bar x) \in p'\}$. This in particular shows that
$\beta_n^{-1}([\phi]) = [\Gamma(\phi)]$ and so $\beta_n$ is continuous. The
inverse interpretation, however, will induce the inverse of $\beta_n$ which
would again be continuous. Therefore $\beta_n$ is a homeomorphism.

To show the naturality of $(\beta_n : n < \omega)$ let $f : n \to m$ be a
function.  If $(a'_0, ..., a'_{m-1}) \in M' \models T'$ realises a type $p \in
\S_m(T')$, then $f^*(\beta_m(p)) = \beta_n(f^*(p))$ is the type of $(a'_{f(0)},
..., a'_{f(n-1)})$ in $\Gamma^*(M')$. Therefore the diagram
\begin{center}
\begin{tikzcd}
\S_m(T') \ar[rr, "\S(T)(f)"] \ar[d, "\beta_m"] && \S_n(T') \ar[d, "\beta_n"] \\
\S_m(T) \ar[rr, "\S(T')(f)"] && \S_n(T)
\end{tikzcd}
\end{center}
commutes, showing the naturality of $(\beta_n : n < \omega)$.

Conversely let $(\beta_n : n < \omega)$ be a natural isomorphism from $\S(T')$
to $\S(T)$. We construct an interpretation $\Gamma$ of $T$ in $T'$ as follows.
For each $n$-ary relation symbol $R$ the set $[R(\bar x)] \in \S_n(T)$ is
compact and open. Since $\beta_n$ is a homeomorphism, so is
$\beta_n^{-1}([R(\bar x)]) \subseteq S_n(T')$. Thus there is a $\Sigma'$-formula
$\phi'(\bar x)$ such that $[\phi'(\bar x)] = \beta_n^{-1}([R(\bar x)])$. The
interpretation $\Gamma$ assigns $\phi'$ to $R$.

The extension of $\Gamma$ to $\Sigma$ can be obtained as follows. Let $\phi(\bar
x) \in \Sigma$ be a formula with $n$-variables. Then $[\phi] \subseteq \S_n(T)$
is compact and open. Therefore as above there is a (unique modulo $T'$)
$\Sigma'$-formula $\phi'(\bar x)$ such that $\beta_n^{-1}([\phi]) = [\phi']$. A
straightforward induction on formulas, details of which we leave to the reader,
shows that $\Gamma(\phi) = \phi'$. This also shows that $\Gamma$ is a bijection
on $\Sigma$ and therefore has an inverse. Finally if $\phi(\bar x), \psi(\bar x)
\in \Sigma$, then 
$$\begin{array}{rl}
T \models \forall \bar x (\phi(\bar x) \to \psi(\bar x)) 
& \text {iff } [\phi] \subseteq [\psi] \\
& \text {iff } \beta_n^{-1}([\phi]) \subseteq \beta_n^{-1}([\psi]) \\
& \text {iff } [\Gamma(\phi)] \subseteq [\Gamma(\psi)] \\
& \text {iff } T' \models \forall \bar x (\Gamma(\phi(\bar x)) \to
\Gamma(\psi(\bar x)))
\end{array}$$
showing that $\Gamma$ is indeed an isomorphic interpretation.
\end{proof}

\begin{remark}
For positively model-complete theories this correspondence is essentially known
(see e.g.  \cite{morley-topology}). It is suggested in \cite{macintyre-aspects}
that type spaces are more fundamental than models themselves. Using the above
correspondence one may choose to work with a type space functor and forget the
h-inductive theory completely.
\end{remark}

\section {Omitting Types}

In Tarskian model theory Henkin construction or model-theoretic forcing is used
to omit types without support. In the Stone space of types these methods become
an instance of the Baire Category Theorem.  This approach is explicitly taken in
\cite{lascar-stability}. Since we have a Baire Category Theorem for compact
sober spaces (Fact \ref{isbell-category}), we would like to adapt the omitting
types theorem to positive model theory. In this section we assume that the
signature we are working with is countable.

First we give a condition for a subset of a model to enumerate a positively
closed model.

\begin{proposition}
Let $M \models T$ and $A \subseteq M$ be a subset satisfying the following
condition. For every $\bar a \in A$ and for every $\phi(\bar x, \bar y)$
positive quantifier-free 
\begin{itemize}
\item either there is $\bar b \in A$ such that $M \models \phi(\bar a, \bar b)$;
\item or there are $\chi(\bar x, \bar z)$ positive quantifier-free and $\bar c
\in A$ such that $M \models \chi(\bar a, \bar c)$ and $T \models \lnot \exists
\bar x, \bar y, \bar z (\phi(\bar x, \bar y) \land \chi(\bar x, \bar z))$.
\end{itemize}
Then $A$ is the universe of a positively closed model of $T$.
\end{proposition}

\begin{proof}
Note that the condition applied to $f(x_0, ..., x_{n-1}) = x_n$ shows that $A$
is closed under interpretations of function symbols. And similarly $A$ contains
the interpretations of constant symbols. So $A$ is the universe of a
substructure.

Now since $A \subseteq M$, we conclude that $A \models T|_\Pi$. By Fact
\ref{pc-fact}, it is a positively closed model of $T|_\Pi$. And so $A$ is also
a positively closed model of $T$ by Fact \ref{companion-fact}
\end{proof}

We need to look at the spaces of types in countably many variables which we
denote by $\S_\omega(T)$. This can be thought of as the space of types in
variables $(x_j : j < \omega)$. It can also be constructed as the inverse limit
of the system $(\S_j(T) : j < \omega)$ with maps $i_{j,k}^* : \S_k(T) \to
\S_j(T)$ where $i_{j,k} = \id_k|_j : j \to k$ for $j \le k$.

\begin{lemma}
\label{comeagreg}
There is a comeagre sets $G \subseteq \S_\omega(T)$ such that if $\bar a \in M
\models T$ realises a type in $G$, then $\bar a$ enumerates a positively closed
model of $T$.
\end{lemma}

\begin{proof}
Let us first introduce a notation. If $\bar i = (i_0, ..., i_{n-1}) \in
\omega^n$, then by $\bar x_{\bar i}$ we denote the tuple $(x_{i_0}, ...,
x_{i_{n-1}})$.

Let $\phi(u_0, ..., u_{n-1}, v_0, ..., v_{m-1})$ be a positive quantifier-free
(pqf for short) formula and $\bar i = (i_0, ..., i_{n-1}) \in \omega^n$.
Consider the following open set
$$O_\phi^{\bar i} = \bigcup_{\bar j \in \omega^m} [\phi(\bar x_{\bar i}, \bar
x_{\bar j})] \cup \bigcup_{\substack{\chi(\bar u, \bar w) \text { pqf} \\ T
\models \lnot \exists \bar u, \bar v, \bar w (\phi(\bar u, \bar v) \land
\chi(\bar u, \bar w))}} \bigcup_{\bar j \in \omega^{|\bar w|}} [\chi(\bar
x_{\bar i}, \bar x_{\bar j})].$$
By the above corollary any tuple realising a type in $G = \bigcap_{\phi \text {
pqf}} \bigcap_{\bar i \in \omega^n} O^{\bar i}_\phi$ enumerates a positively
closed model of $T$. It remains to show that each $O^{\bar i}_\phi$ is dense.

Consider a nonempty basic open set in $\S_\omega(T)$. By adding dummy variables
if necessary we may assume that it has the form $[\exists \bar z \psi(\bar
x_{\bar i}, \bar x_{\bar k}, \bar z)]$, where $\psi$ is positive quantifier-free
and $\bar k$ does not intersect $\bar i$. Now if $\exists \bar x_{\bar i} \bar v
\bar x_{\bar k} \bar z (\phi(\bar x_{\bar i}, \bar v) \land \psi(\bar x_{\bar
i}, \bar x_{\bar k}, \bar z))$ is consistent with $T$, then $[\exists \bar z
\psi(\bar x_{\bar i}, \bar x_{\bar k}, \bar z)]$ intersects $[\phi(\bar x_{\bar
i}, \bar x_{\bar j})]$ for some tuple $\bar j \in \omega^m$. But otherwise
$\psi(\bar u, \bar y \bar z)$ is one of formulas $\chi$ in the right hand union.
Thus $[\exists \bar z \psi(\bar x_{\bar i}, \bar x_{\bar k}, \bar z)]$ again
intersects the open set $O^{\bar i}_\phi$.
\end{proof}

Using this lemma we can prove an omitting types theorem in positive model
theory.

\begin{theorem}
\label{omittingmeagretypes}
Let $A_n \subseteq \S_n(T)$ be a meagre set of $n$-types for each $n$. Then
there is a positively closed model of $T$ that omits every type in each
$A_n$.
\end{theorem}

\begin{proof}
The preimage of a meagre set under a continuous open map is meagre. Therefore $E
= \bigcup\limits_{n < \omega} \bigcup\limits_{f : n \to \omega}
(f^*)^{-1}(A_n) \subseteq \S_\omega(T)$ is meagre. By the previous Lemma $G \cap
(\S_\omega(T) \setminus E)$ is comeagre. Now any element in that set will
enumerate a positively closed model of $T$ that omits all types in
$\bigcup\limits_{n < \omega} A_n$. To finish we observe that any comeagre set
is dense and so nonempty in a compact sober space by Isbell's theorem (Fact
\ref{isbell-category}).
\end{proof}

This theorem can be formulated in more familiar terms without any mention of
topology. Let $p(\bar x)$ be a partial $\Pi$-type. The $\Sigma$-formula
$\psi(\bar x)$ is said to be a {\em support} for $p$ if it is consistent with
$T$ and $T \models \forall \bar x (\psi(\bar x) \to \phi(\bar x))$ for every
$\phi \in p$. Note that any closed set $A \subseteq S_n$ is of the form $A =
[p]$ for some partial $\Pi$-type $p$. Then $[p]$ having an empty interior (i.e.
being nowhere dense) is equivalent to $p$ having no support. Thus given a
countable set $\{p_i : i < \omega\}$ of $\Pi$-types in $S_n(T)$ without support,
the set $\bigcup_{i < \omega} [p_i]$ is meagre and conversely any meagre
subset of $S_n(T)$ is contained in a subset of such form. Thus we have another
formulation of the omitting types theorem.

\begin{corollary}
\label{omittingtypes}
Let $p_i$ be a partial $\Pi$-type without support for each $i \in \omega$. Then
$T$ has a positively closed model omitting them all.
\end{corollary}

\begin{remark}
Alex Kruckman has pointed out that constructing  positively models can
be viewed as an instance of omitting types. Indeed a model of $T$ is
positively closed if and only if it omits all types whose $\Sigma$-parts are not
maximal. But $p|_\Sigma$ is not maximal if and only if there is $\phi \in \Sigma
\setminus p$ such that $p|_\Sigma \cup \{\phi\}$ is consistent. This is
equivalent to $p \in \overline {[\phi]} \setminus [\phi]$ which is nowhere
dense. Thus for any $n$, the set of types in $\S_n$ whose $\Sigma$ parts are not
maximal is $\bigcup_{\phi \in \Sigma} \overline{[\phi]} \setminus [\phi]$ which
is meagre.

This may also be used to simplify the proof of the omitting types theorem, if we
can show that the set $\{\tp(\bar a) \in \S_\omega(T) : \bar a \text {
enumerates a model of $T$}\}$ is comeagre. But it is not clear how to do this
other than using Lemma \ref{comeagreg}.
\end{remark}

\section {Atomic Models}

In this section we study atomic positively closed models of an h-inductive
theory $T$. We continue to assume that the signature is countable. Countable
throughout this section includes finite.

\begin{definition}
A positively closed model $M$ of $T$ is called {\em atomic} if every
$\Pi$-type realised in $M$ is supported.
\end{definition}

We will show that the theory of atomic models in positive logic generalises the
theory of atomic models in Tarskian model theory. We'll start by characterising
countable and atomic positively closed models.

\begin{definition}
A positively closed model $M \models T$ is called {\em prime} if every
positively closed model $N \models T$ is a continuation of $M$.
\end{definition}

\begin{proposition}
A positively closed model of an h-inductive theory with the joint continuation
property is prime if and only if it is countable and atomic.
\end{proposition}

\begin{proof}
Assume $M$ is prime. Since homomorphisms from positively closed models are
injective and $T$ has countable positively closed models, $M$ must itself be
countable. If $\bar a \in M$ realises a $\Pi$-type without support, then there
is a positively closed model $N \models T$ omitting it and $N$ cannot be a
continuation of $M$. Therefore $M$ must be atomic.

Conversely assume that $M$ is countable and atomic. Let $(a_i : i < \omega)$
be an enumeration of $M$ possibly with repetitions (since $M$ may be finite).
Let $N \models T$ be positively closed. We define a homomorphism $f : M \to
N$ inductively. 

Assume that we have already defined $f$ on $\{a_0, ..., a_{n-1}\}$ and
it satisfies $\tp^M_{\Sigma \cup \Pi}(a_0, ..., a_{n-1}) = \tp^N_{\Sigma \cup
\Pi}(f(a_0), ..., f(a_{n-1}))$. Note that for $n=0$ this condition is true by
the joint continuation property. Now since $M$ is atomic, there is a formula
$\phi(x_0, ..., x_n)$ that is a support for $\tp^M_\Pi(a_0, ..., a_n)$. By the
assumption $N \models \exists x_n \phi(f(a_0),...,f(a_{n-1}), x_n)$. Let $f(a_n)
\in N$ be an element such that $N \models \phi(f(a_0),...,f(a_n))$. Then
$\tp^M_\Pi(a_0, ..., a_n) \subseteq \tp^N_\Pi(f(a_0), ..., f(a_n))$. But since
both $M$ and $N$ are positively closed it follows that $\tp^M_{\Sigma \cup
\Pi}(a_0, ..., a_n) = \tp^N_{\Sigma \cup \Pi}(f(a_0), ..., f(a_n))$.
\end{proof}

We can do the above construction of a homomorphism in a back and forth manner to
obtain.

\begin{corollary}
Any two countable positively closed atomic models of an h-inductive
theory with the joint continuation property are isomorphic.
\end{corollary}

This allows us to have a characterisation of countably categorical h-inductive
theories.

\begin{theorem}
\label{countcat}
For an h-inductive theory $T$ with the joint continuation property the following
are equivalent.
\begin{enumerate}
\item All positively closed models of $T$ are atomic.
\item $T$ is countably categorical, i.e. any two (at most) countable
positively closed models are isomorphic.
\item For any $n < \omega$ each irreducible component of $\S_n(T)$ has nonempty
interior.
\end{enumerate}
\end{theorem}

\begin{proof}
$1 \implies 2$ is the previous corollary.

$2 \implies 3$. Assume that $3$ fails. Since irreducible components correspond
to minimal $\Pi$-types, there is a minimal $\Pi$-type $p$ in $n$-variables
without support. By the omitting types theorem there is a countable
positively closed model $M$ that omits $p$. On the other hand there is a
countable positively closed model $N$ that realises $p$ (any countable
positively closed continuation of a model realising $p$ will do). It is clear
that $M \not \cong N$.

$3 \implies 1$. Let $N \models T$ be positively closed and $\bar a \in N$.
Then $p = \tp_\Pi(\bar a)$ is a minimal $\Pi$-type, so $[p] = \{q \in \S_n(T) :
q \supseteq p\}$ is an irreducible component of $\S_n(T)$. By the hypothesis it
has a nonempty interior. So there is a basic open set $[\phi(\bar x)]$ such that
$[\phi(\bar x)] \subseteq [p]$. Then $\phi$ is a support for $p$.
\end{proof}

Unlike Hausdorff spaces, however, the last condition does not imply that there
are finitely many irreducible components. Correspondingly a countably
categorical h-inductive theory may have infinitely many types and infinitely
many definable formulas.

\begin{example}
Assume that the language has constant symbols $c_i$ for each $i < \omega$.
Consider the h-inductive theory $T$ asserting that they are distinct, i.e. for
$i \neq j$ the axiom $c_i = c_j \to \bot$ is in $T$. This theory has a unique
positively closed model consisting of only the interpretations of $c_i$.
Therefore $T$ is countably categorical. But $x = c_i$ define infinitely many
distinct sets.
\end{example}

\begin{example}
We can modify the previous example to make sure atomic formulas have negations.
This will correspond to a first-order indutive theory, whose class of
existentially closed models is countably categorical but there are infinitely
many definable set. Such an example was asked in \cite{simmons-large-small}. We
can also make sure this theory has the amalgamation property.

Firstly we enlarge the language to include new relation $\neq$ and add axioms to
$T$ expressing that this is the negation of $=$. Now an existentially closed
model can have an arbitrary number of elements besides interpretations of
constant symbols. To make sure this theory is countably categorical, we need to
ensure that there are infinitely many of these non-constant elements.

So we add a binary relation $Q(x, y)$ (and its negation $\lnot Q$) that will
pair these extra elements. The theory asserts the following.
\begin{itemize}
\item $Q$ is symmetric and irreflexive.
\item Each element has at most one pair, and only non-constant elements have
pairs. 
\item There are infinitely many pairs.
\end{itemize}

In an arbitrary model of $T$ only some of the non-constant elements have pairs.
But in an extension elements can always become paired. So there is a unique
existentially closed model where each non-constant element has a pair. It is
easy to verify that $T$ has the amalgamation property.
\end{example}

It is also worth mentioning here that the h-inductive theory of the structure
$(\mathbb N; \le, 0, 1, ...)$ has precisely two positively closed models. The
other model has an infinite element larger than all natural numbers. Being a
theory of a structure this theory also has the joint continuation property. This
example appears in \cite{poizat-quelques}. In contrast a well-known result of
Vaught states that a complete first-order theory cannot have exactly two
countable models.

Recall that in Tarskian model theory a complete theory $T$ has an atomic
model if and only if the set of isolated points is dense in each $\S_n(T)$. It
turns out that in the generality of our setting, we need to look at the set of
somewhere dense points. I.e. points $p \in \S_n(T)$ such that $\overline p$ has
nonempty interior. Of course in a Hausdorff space points are closed, so $p$ is
somewhere dense if and only if it is isolated. In general however, being
somewhere dense only implies that $\overline p$ is an irreducible component.
Indeed if $\overline p \subsetneq A$ and $O \subseteq \overline p$ is open, then
$A = (A \setminus O) \cup \overline p$ is reducible.

\begin{theorem}
\label{denseexists}
If for every $n < \omega$ the set of somewhere dense points is dense in
$\S_n(T)$, then $T$ has an atomic positively closed model. The converse is
true if $T$ has the joint continuation property.
\end{theorem}

\begin{proof}
Let $A_n \subseteq \S_n(T)$ be the set of points that are somewhere dense.
Assume that $A_n$ is dense in $\S_n(T)$ for every $n$. Define
$$q_n = \{\lnot \phi(\bar x) \in \Pi_n: [\phi] \subseteq \overline p \text { for
some } p \in A_n\}.$$
Then $[q_n] \cap A_n = \emptyset$. Indeed if $p \in A_n$, then $\overline p$ has
a nonempty interior. So let $\chi(\bar x) \in \Sigma$ be such that $[\chi]
\subseteq \overline p$. Then $\chi \in p$ and $\lnot \chi \in q_n$. So $p \not
\in [q_n]$.  Since $A_n$ is dense, we conclude that $[q_n]$ must be nowhere
dense so that $q_n$ has no support.

By the omitting types theorem there is a positively closed model $M \models
T$ that omits $q_n$ for every $n$. Let $\bar b \in M$. Since $M$ omits $q_n$,
there is formula $\phi(\bar x) \in \tp_\Sigma(\bar b)$ such that $[\phi]
\subseteq \overline p$ for some $p \in A_n$. But then $\bar b$ realises
$p$ proving that $M$ is indeed atomic.

Conversely assume that $T$ has the joint continuation property and let $M
\models T$ be positively closed and atomic. Fix $n < \omega$ and take a
nonempty basic open set $[\phi(\bar x)]$ (where $\phi \in \Sigma$). Then $M
\models \exists \bar x \phi(\bar x)$ (by the joint continuation property), so
let $\bar b \in M$ realise $\phi(\bar x)$. Then $\tp_{\Sigma \cup \Pi}(\bar b)$
is somewhere dense and is contained in $[\phi(\bar x)]$.
\end{proof}

Finally we show that counting irreducible components of $\S_n(T)$ (as opposed to
just points, as in Hausdorff spaces) can be used to show the existence of atomic
positively closed models.

\begin{theorem}
\label{countableexists}
If for every $n$ the space $\S_n(T)$ has countably many irreducible components,
then $T$ has an atomic positively closed model.
\end{theorem}

\begin{proof}
We show that the set $A$ of somewhere dense points is dense in $\S_n(T)$. Let
$B$ be the set of generic types of irreducible components. By the assumption $B$
is countable and $\overline B = \S_n(T)$. If $p \in B \setminus A$, then
$\S_n(T) \setminus \overline p$ is dense open. Hence by Isbell's category
theorem the set $$C = \S_n(T) \setminus \bigcup_{p \in B \setminus A} \overline
p$$ is also dense. Now let $q \in C$ and let $p \in B$ be the generic type of
some irreducible component containing $q$. Then $p \in A$ and so $q \in
\overline p \subseteq \overline A$. Thus $C \subseteq \overline A$ proving that
$A$ is also dense in $\S_n(T)$.
\end{proof}

\section* {Acknowledgements}

The author is grateful to Alex Kruckman and the anonymous referee, who made a
number of suggestions on improving the presentation of the paper.

\bibliographystyle{plainnat}
\bibliography{../all}

\end{document}